\documentclass[11pt,a4paper,reqno,twoside]{article}
\usepackage{hyperref}
\usepackage{amsfonts}
\usepackage{bbding}
\usepackage[all]{xy}
\usepackage{young}
\usepackage{multirow}
\usepackage{booktabs}
\usepackage{amsmath,amscd,amssymb,latexsym,srcltx,indentfirst,titlesec}
\usepackage{amsthm}

\usepackage{enumitem}

\setitemize[1]{itemsep=0pt,partopsep=0pt,parsep=\parskip,topsep=5pt}

\numberwithin{equation}{section}
\newtheorem{theorem}{Theorem}[section]
\newtheorem{lemma}[theorem]{Lemma}
\newtheorem*{thma*}{Theorem A}
\newtheorem*{thmb*}{Theorem B}
\newtheorem*{thmc*}{Theorem C}
\newtheorem*{thmd*}{Theorem D}
\newtheorem*{thme*}{Theorem E}

\newtheorem{corollary}[theorem]{Corollary}

\newtheorem{definition}[theorem]{Definition}

\newtheorem*{thm*}{Main Theorem}
\allowdisplaybreaks

\begin{document}

\makeatletter

\newdimen\bibspace
\setlength\bibspace{2pt}   % 就是修改这个地方啦
\renewenvironment{thebibliography}[1]{%
 \section*{\refname %or \bibname if you use ``book'' as the documentclass
       \@mkboth{\MakeUppercase\refname}{\MakeUppercase\refname}}%
     \list{\@biblabel{\@arabic\c@enumiv}}%
          {\settowidth\labelwidth{\@biblabel{#1}}%
           \leftmargin\labelwidth
           \advance\leftmargin\labelsep
           \itemsep\bibspace
           \parsep\z@skip     %
           \@openbib@code
           \usecounter{enumiv}%
           \let\p@enumiv\@empty
           \renewcommand\theenumiv{\@arabic\c@enumiv}}%
     \sloppy\clubpenalty4000\widowpenalty4000%
     \sfcode`\.\@m}
    {\def\@noitemerr
      {\@latex@warning{Empty `thebibliography' environment}}%
     \endlist}

\makeatother

\pdfbookmark[2]{Control of fusion by elementary abelian subgroups of rank at least 2}{beg}

\title{{\bf Control of fusion by elementary abelian subgroups of rank at least 2} \footnotetext{\hspace*{-4 ex}
1. School of Mathematical Sciences, Peking University, Beijing, 100871, China\\
2. Department of Mathematics, Hubei University, Wuhan, $420062$, China\\
Supported by National Key R \& D Program of China(Grant No.2020YFE0204200) and NSFC grants (11771129)\\
}}

\author{{\small{Lizhong Wang$^1$, Xingzhong Xu$^2$, Jiping Zhang$^1$} }
}

\date{}
\maketitle

{\small

\noindent\textbf{Abstract.} {\small{In this paper, we focus on the subgroups control $p$-fusion, and we improve the
Theorem B of \cite{BGH} for odd prime. For odd prime, we prove that
elementary abelian subgroups of rank at least 2 can control $p$-fusion(see our Theorem B).
 }}\\


\noindent\textbf{Mathematics Subject Classification (2010):}  }

\section{\bf Introduction}

Burnside's fusion theorem tell us that if a finite group $G$ has an abelian Sylow $p$-subgroup $S$,
the morphisms of $\mathcal{F}_{S}(G)$ can be controlled by $\mathrm{Aut}_{G}(S)$. "control" means each morphism of $\mathcal{F}_{S}(G)$ can be written from $\mathrm{Aut}_{G}(S)$
with composition and inclusion.

What's about non-abelian $p$-group $S$?
There is a remarkable theorem which is named as Model theorem for constrained fusion systems tell us that
if $Q$ is normal in the saturated fusion system $\mathcal{F}$ and $Q$ is $\mathcal{F}$-centric, each morphism of $\mathcal{F}_{S}(G)$ can be written from $\mathrm{Aut}_{\mathcal{F}}(Q)$
with composition and inclusion. In other words, there exists a finite group $G$ such that $\mathcal{F}=\mathcal{F}_S(G)$ and $\mathrm{Aut}_{\mathcal{F}}(Q)=\mathrm{Aut}_{G}(Q)$.
Burnside's theorem can be seem as an example of the Model theorem.

In \cite{BGH}, Benson, Grodal and Henke found that small exponent abelian $p$-subgroups can control $p$-fusion, and they proved the following theorem.
\begin{thma*}\cite[Theorem B]{BGH}(Small exponent abelian p-subgroups control $p$-fusion)
 Let $\mathcal{G} \leq  \mathcal{F}$ be two saturated fusion systems on the same finite p-group $S$. Suppose that  $\mathrm{Hom}_{\mathcal{G}}(A,B) = \mathrm{Hom}_{\mathcal{F}}(A,B)$
  for all $A,B \leq S$ with $A,B$ elementary abelian if $p$ is odd, and abelian of exponent at most 4 if $p = 2$. Then
$\mathcal{G}= \mathcal{F}$.
\end{thma*}

In our paper, we prove an stronger version of above theorem for odd prime.
\begin{thmb*} Let $p$ be an odd prime. Let $\mathcal{G} \leq  \mathcal{F}$ be two saturated fusion systems on finite $p$-group $S$ and $S$ be not cyclic.
Suppose that $\mathrm{Hom}_{\mathcal{G}}(A,B) = \mathrm{Hom}_{\mathcal{F}}(A,B)$ for all $A,B \leq S$ where $A,B$ are elementary abelian subgroups of rank at least 2. Then
$\mathcal{G}= \mathcal{F}$.
\end{thmb*}

$Structure~ of ~ the~ paper:$ After recalling the basic definitions and properties of fusion systems in Section 2,
and we prove the Theorem B in Section 3.

\section{\bf Fusion systems}

In this section we collect some known results that will be needed later.  For the background theory of
fusion systems, we refer to \cite{AsKO}. For notations, we mainly reference \cite{A}.

\begin{definition} A $fusion ~ system$ $\mathcal{F}$ over a finite $p$-group
$S$ is a category whose objects are the subgroups of $S$, and whose
morphism sets $\mathrm{Hom}_{\mathcal{F}}(P,Q)$ satisfy the
following two conditions:

\vskip 0.3cm

(a) $\mathrm{Hom}_{S}(P,Q)\subseteq
\mathrm{Hom}_{\mathcal{F}}(P,Q)\subseteq\mathrm{Inj}(P,Q)$ for all
$P,Q\leq S$.

\vskip 0.3cm

(b) Every morphism in $\mathcal{F}$ factors as an isomorphism in
$\mathcal{F}$ followed by an inclusion.

\end{definition}

\begin{definition} Let $\mathcal{F}$ be a fusion system over a $p$-group
$S$.

$\bullet$ Two subgroups $P,Q$ are $\mathcal{F}$-$conjugate$ if they
are isomorphic as objects of the category $\mathcal{F}$. Let
$P^{\mathcal{F}}$ denote the set of all subgroups of $S$ which are
$\mathcal{F}$-conjugate to $P$.

$\bullet$ A subgroup $P\leq S$ is $fully~automised$ in $\mathcal{F}$
if $\mathrm{Aut}_{S}(P)\in
\mathrm{Syl}_{p}(\mathrm{Hom}_{\mathcal{F}}(P,P)=\mathrm{Aut}_{\mathcal{F}}(P))$.

$\bullet$ A subgroup $P\leq S$ is $receptive$ in $\mathcal{F}$ if it
has the following property: for each $Q\leq S$ and each $\varphi\in
\mathrm{Iso}_{\mathcal{F}}(Q, P)$, if we set
$$N_{\varphi}=N_{\varphi}^{\mathcal{F}}=\{g\in N_{S}(G)|\varphi \circ c_{g}\circ \varphi^{-1}\in \mathrm{Aut}_{S}(P)\},$$
then there is $\overline{\varphi}\in
\mathrm{Hom}_{\mathcal{F}}(N_{\varphi},S)$ such that
$\overline{\varphi}|_{Q}=\varphi$.

$\bullet$ A fusion system $\mathcal{F}$ over a $p$-group $S$ is
$saturated$ if each subgroup of $S$ is $\mathcal{F}$-conjugate to a
subgroup which is fully automised and receptive.
\end{definition}

\begin{definition} Let $\mathcal{F}$ be a fusion system over a $p$-group
$S$.

$\bullet$ A subgroup $P\leq S$ is $fully~centralised$ in
$\mathcal{F}$ if $|C_{S}(P)|\geq |C_{S}(Q)|$ for all $Q\in
P^{\mathcal{F}}$.

$\bullet$ A subgroup $P\leq S$ is $fully~normalised$ in
$\mathcal{F}$ if $|N_{S}(P)|\geq |N_{S}(Q)|$ for all $Q\in
P^{\mathcal{F}}$.

$\bullet$ A subgroup $P\leq S$ is $\mathcal{F}$-$centric$ if
$C_{S}(Q)=Z(Q)$ for $Q\in P^{\mathcal{F}}$.

$\bullet$ A subgroup $P\leq S$ is $\mathcal{F}$-$radical$ if
$\mathrm{Out}_{\mathcal{F}}(P)=\mathrm{Aut}_{\mathcal{F}}(P)/\mathrm{Inn}(P)$
is $p$-reduced; i.e., if $O_{p}(\mathrm{Out}_{\mathcal{F}}(P))=1$.

$\bullet$ A subgroup $P\leq S$ is $normal$ in $\mathcal{F}$ (denoted
$P\trianglelefteq \mathcal{F}$) if for all $Q,R\in S$ and all
$\varphi\in \mathrm{Hom }_{\mathcal{F}}(Q,R)$, $\varphi$ extends to
a morphism $\overline{\varphi}\in \mathrm{Hom
}_{\mathcal{F}}(QP,RP)$ such that $(P)\overline{\varphi}=P$.

\end{definition}

\begin{theorem}(Alperin's theorem for fusion systems). Let $\mathcal{F}$ be a saturated fusion system over a $p$-group
$S$. Then for each morphism $\varphi\in \mathrm{Iso}_{\mathcal{F}}(P,R)$ in
$\mathcal{F}$, there exist sequences of subgroups of $S$
$$P=P_{0},P_{1},\ldots,P_{k}=R,~~~~~ \mathrm{and }~~~~~ Q_{1},Q_{2},\ldots,Q_{k},$$
and morphism $\psi_{i}\in \mathrm{Aut}_{\mathcal{F}}(Q_{i})$, such
that

$(a)$ $Q_{i}$ is fully~normalised in $\mathcal{F}$,
$\mathcal{F}$-radical, and $\mathcal{F}$-centric for each $i$;

$(b)$ $P_{i-1},P_{i}\leq Q_{i}$ and $\psi_{i}(P_{i-1})=P_{i}$ for
each $i$; and

$(c)$ $\psi=\psi_{k}\circ\psi_{k-1}\circ\cdots\circ\psi_{1}$.

\end{theorem}

The following theorems are the key steps to prove the Theorem B.

\begin{theorem}\cite[Theorem 2.1]{BGH} Let $p$ be an odd prime. Let $P$ be a finite $p$-group. There exists a characteristic subgroup $D$ of $P$, of
exponent p, such that $[D,P] \leq Z(D)$, and
such that every non-trivial $p'$-automorphism of $P$ restricts to a non-trivial automorphism of $D$. Furthermore, for any such $D$ and any maximal (with
respect to inclusion) abelian subgroup $A$ of $D$ we have $A \unlhd P$ and $C_{\mathrm{Aut}(P)}(A)$
is a $p$-group.
\end{theorem}

\begin{theorem}\cite[Main Lemma 2.4]{BGH} Let $\mathcal{G} \leq  \mathcal{F}$ be two saturated fusion systems on the same
finite $p$-group $S$, and $P \leq S$ an $\mathcal{F}$-centric and fully $\mathcal{F}$-normalized subgroup, with $\mathrm{Aut}_{\mathcal{F}}(R)=\mathrm{Aut}_{\mathcal{G}}(R)$ for every
$P<R \leq N_S(P)$. Suppose that
there exists a subgroup $Q\unlhd P$ with $\mathrm{Hom}_{\mathcal{F}}(Q,S) = \mathrm{Hom}_{\mathcal{G}}(Q,S).$ Then
$\mathrm{Aut}_{\mathcal{F}}(P)=
\langle \mathrm{Aut}_{\mathcal{G}}(P),C_{\mathrm{Aut}_{\mathcal{F}}(P)}(Q)\rangle$.
\end{theorem}

Recall that $cl(C)$ is the nilpotent class of finite group $C$.
\begin{lemma} Let $p$ be an odd prime. Let $C$ be a finite $p$-group and $C$ is not cyclic. If $cl(C)\leq 2$ and $C/Z(C)$ is elementary,
Then $rk_p(\Omega_1(C))\geq 2$.
\end{lemma}

\begin{proof}
If $C$ is abelian, it's nothing need to prove. Let $a, b\in C$ such that $[a, b]\neq 1$. Since  $cl(C)\leq 2$, we have $[a, b]=c\in Z(C)$.
Set $c:=[a, b]$ and $|c|=p^r$ for some integer. We shall suppose that $a, b$ so chosen that $|a|+|b|+|[a, b]|$is minimal subject to $[a, b]\neq 1$.
Since $cl(C)\leq 2$, we have
$[a, b^{p^{r-1}}]=c^{r-1}$. Now, we can assume that $[a, b]=c$ and $|c|=p$. Set $|a|=p^n, |b|=p^{m}$ for some integers $n$ and $m$.
We can see that $a^{p^{n-1}}, b^{p^{m-1}}, c\in \Omega_1(C)$. If $a^{p^{n-1}}\notin \langle c\rangle$ or $b^{p^{m-1}}\notin \langle c\rangle$, then $rk_p(\Omega_1(C))\geq 2$.
Suppose that $a^{p^{n-1}}, b^{p^{m-1}}\in \langle c\rangle$, we can assume that $a^{p^{n-1}}=b^{p^{m-1}}=c$ and $n\geq m$. We can see that $(a^{n-m}b^{-1})^{p^{m-1}}=1$
and $[a, a^{n-m}b^{-1}]=c^{-1}$, that is a contradiction to the minimal chosen.
\end{proof}

\section{\bf Proof of Theorem B}

{\bf Proof of Theorem B}~~
By Alperin's fusion theorem, $\mathcal{F}$ is generated by $\mathcal{F}$-automorphisms of fully $\mathcal{F}$-normalized and $\mathcal{F}$-centric subgroups; see [2, Theorem I.3.6] (in fact we only need "$\mathcal{F}$-essential" subgroups and $S$). We want to
show that $\mathrm{Aut}_{\mathcal{G}}(P) = \mathrm{Aut}_{\mathcal{F}}(P)$ for all $P\leq S$; by downward induction on the
order we can assume that $\mathrm{Aut}_{\mathcal{G}}(R) = \mathrm{Aut}_{\mathcal{F}}(R)$ for all subgroups $R \leq S$ with
$|R| > |P|$, and by the fusion theorem we can furthermore assume that $P$ is $\mathcal{F}$-centric and fully $\mathcal{F}$-normalized.

{\bf Step 1.} If $P=S$, since $S$ is not cyclic, by \cite[Theorem 3.11]{Go}, $P(=S)$ possesses a characteristic subgroup $C$ with the following properties:
(i) $cl(C)\leq 2$ and $C/Z(C)$ is elementary abelian;
(ii) $[P, C]\leq Z(C)$;
(iii) $C_P(C)\leq Z(C)$.

{\bf ~Case 1.} If $C$ is cyclic, we can set $C=\langle b\rangle$. Since $P/C=N_P(C)/C_P(C)\leq \mathrm{Aut}(C)$ and $C$ is cyclic, we have $P/C$
is cyclic because $\mathrm{Aut}(C)$ is cyclic(see \cite[Exe. 25]{B1}). So, we can set $P=\langle a, b\rangle$ and $P/C=\langle aC\rangle$.

Since $[P, C]\leq C$, we can suppose that
$$[a, b]=b^{p^t}.$$
Here, we can see that $P':=[P, P]=\langle b^{p^t} \rangle$ because $P=\langle a, b\rangle$.
Set $|a|:=p^n$ and $|b|:=p^m$ for some inters $n,m$. Set $A=\langle a^{p^{n-1}}, b^{p^{m-1}}\rangle=\Omega_1(P)$, we can see that $A$
is a normal elementary abelian subgroup of $P$ and $rk_p(A)=2$. So, by \cite[Main Lemma 2.4, Step 1]{BGH}, we have $\mathrm{Aut}_{\mathcal{G}}(P) = \mathrm{Aut}_{\mathcal{F}}(P)$.

{\bf ~Case 2.} If $C$ is not cyclic, set $D=\Omega_1(C)$. So, we have
$D$ is not cyclic when $p$ is odd. Then we have the rank of $D$ is more that 2.
So, by \cite[Main Lemma 2.4, Step 1]{BGH}, we have $\mathrm{Aut}_{\mathcal{G}}(P) = \mathrm{Aut}_{\mathcal{F}}(P)$.

{\bf Step 2.} If $P\lneq S$, we will prove the theorem in the following two cases.

{\bf ~Case 1.} If $P$ is cyclic, we can set $P=\langle a\rangle$. Since $P$ is centric, we have $C_{S}(P)\leq P$.
We can assume that $N_S(P)\gneq P$ because $S$ is not cyclic. For each $g\in N_S(P)$, we can set
$$a^g=a^t$$
for some integer $t$.
Let $\alpha\in \mathrm{Aut}_{\mathcal{F}}(P)$, we can set
$$\alpha(a)=a^i$$
for some integer $i$.
Hence, $$\alpha^{-1}\circ c_g\circ \alpha(a)=\alpha^{-1}\circ c_g(a^i)=\alpha^{-1}(a^{it})=\alpha^{-1}(a^{i})^t=a^t=c_g(a).$$
So, we have $\mathrm{Aut}_S(P)\unlhd \mathrm{Aut}_{\mathcal{F}}(P).$
But we assume that $N_S(P)\gneq P$, thus $P$ is not radical. That is a contradiction because that $P$ is $\mathcal{F}$-essential.

{\bf ~Case 2.} If $P$ is not cyclic, by \cite[Theorem 3.11]{Go}, $P$ possesses a characteristic subgroup $C$ with the following properties:
(i) $cl(C)\leq 2$ and $C/Z(C)$ is elementary abelian;
(ii) $[P, C]\leq Z(C)$;
(iii) $C_P(C)\leq Z(C)$.

{\bf ~Case 2.1.} If $C$ is cyclic, we assert that $P$ is not radical.
Since $C$ is cyclic, we can set $C=\langle b\rangle$. Since $P/C=N_P(C)/C_P(C)\leq \mathrm{Aut}(C)$ and $C$ is cyclic, we have $P/C$
is cyclic because $\mathrm{Aut}(C)$ is cyclic(see \cite[Exe. 25]{B1}). So, we can set $P=\langle a, b\rangle$ and $P/C=\langle aC\rangle$.

Since $[P, C]\leq C$, we can suppose that
$$[a, b]=b^{p^t}.$$
Here, we can see that $P':=[P, P]=\langle b^{p^t} \rangle$ because $P=\langle a, b\rangle$.
Now, we want to prove that $\mathrm{Aut}_S(P)\unlhd \mathrm{Aut}_{\mathcal{F}}(P)$.
We can construct a morphism form $\mathrm{Aut}_{\mathcal{F}}(P)$ to $\mathrm{Aut}(P/\Phi(P))$ as follows:
$$\phi: \mathrm{Aut}_{\mathcal{F}}(P)\longrightarrow \mathrm{Aut}(P/\Phi(P))$$
$$\alpha\longmapsto \overline{\alpha}$$
where $\overline{\alpha}(g\Phi(P))=\alpha(g)\Phi(P)$ for each $g\in P$. Here, $\Phi(P)$ is the Frattini subgroup of $P$.
By \cite[24.1]{A}, we can see that $\mathrm{Ker}(\phi)$ is a $p$-group.

Since $P=\langle a, b\rangle$, we have $\mathrm{Aut}(P/\Phi(P))\leq \mathrm{GL}_2(\mathbb{F}_p)$.
For each $\alpha\in \mathrm{Aut}_{\mathcal{F}}(P)$. Since $C$ is a characteristic subgroup of $P$,
we can set that
$$\alpha(a)=a^{r_1}b^{r_2}; ~\alpha(b)=b^{r_3}$$
where $r_1,r_2,r_3$ are integers.
So, we have an injective morphism as follows:
$$\overline{\phi}: \mathrm{Aut}_{\mathcal{F}}(P)/\mathrm{Ker}(\phi)\hookrightarrow \mathrm{Aut}(P/\Phi(P))\leq \mathrm{GL}_2(\mathbb{F}_p)$$
$$\alpha\longmapsto  \left(%
\begin{array}{cc}
\overline{r}_1 & 0 \\
\overline{r}_2 & \overline{r}_3 \\
\end{array}%
\right)$$
where $\overline{r}_i \in \mathbb{F}_p$ and $\overline{r}_i \equiv r_i ~\mathrm{mod}~p$ for $i=1,2,3.$

Since the subgroup $\{\left(%
\begin{array}{cc}
1 & 0 \\
\ast & 1 \\
\end{array}%
\right)\}$ is a normal Sylow $p$-subgroup of the group $\{\left(%
\begin{array}{cc}
\ast & 0 \\
\ast & \ast \\
\end{array}%
\right)\}$ in $\mathrm{GL}_2(\mathbb{F}_p)$,
we have the Sylow $p$-subgroup $\mathrm{Aut}_{\mathcal{F}}(P)/\mathrm{Ker}(\phi)$ is a normal subgroup of $\mathrm{Aut}_{\mathcal{F}}(P)/\mathrm{Ker}(\phi)$.
So, the Sylow $p$-subgroup $\mathrm{Aut}_{\mathcal{F}}(P)$ is a normal subgroup of $\mathrm{Aut}_{\mathcal{F}}(P)$ because $\mathrm{Ker}(\phi)$ is a $p$-group.
Here, $P$ is fully-normalized in $\mathcal{F}$, we have $\mathrm{Aut}_S(P)$ is a Sylow $p$-subgroup of $\mathrm{Aut}_{\mathcal{F}}(P)$.
Hence, $\mathrm{Aut}_S(P)\unlhd \mathrm{Aut}_{\mathcal{F}}(P)$.
That's means $P$ is not radical when $P\neq S$.

{\bf ~Case 2.2.} If $C$ is not cyclic, set $D=\Omega_1(C)$. So, we have
$D$ is not cyclic when $p$ is odd. Then we have the rank of $D$ is more that 2.
By \cite[Theorem 2.1]{BGH}, we can choose that maximal abelian subgroup $A$ of $D$.
Since the rank of $D$ is more that 2, we have the rank of $A$ is more than 2.
By the condition, we have $\mathrm{Aut}_{\mathcal{G}}(A)=\mathrm{Aut}_{\mathcal{F}}(A)$.
Hence, we can conclude that $\mathrm{Aut}_{\mathcal{G}}(P)= \mathrm{Aut}_{\mathcal{F}}(P)$
by the similar proof of \cite[Theorem B]{BGH}. ~~~~~~~~~~~~~~~~~~~~~~~~~~~~~~~~~~~~~~~$\square$

~

Similar to \cite[Theorem D]{BGH}, and using the above theorem we can get the following corollary.

\begin{corollary} Let $p$ be an odd prime. Let $T$ and $S$ be finite $p$-groups and they are not cyclic.
Let $\mathcal{F}$ and $\mathcal{G}$ be saturated fusion systems on finite $p$-groups $S$ and
$T$ respectively, and that $\phi: T\to S$ is a fusion preserving homomorphism
inducing a bijection $\mathrm{Rep}(A,G)\to \mathrm{Rep}(A,F)$ for any finite abelian $p$-group
$A$ with $2\leq rk_p(A) \leq rk_p(S)$. Then $\phi$ induces an isomorphism from $T$ to $S$ and $\mathcal{G}$
to $\mathcal{F}$.
\end{corollary}

Here, $\mathrm{Rep}(A,G)$ denotes the quotient of $\mathrm{Hom}(A,G)$ where we identify
$\phi$ with $c_g\circ \phi$ for all $g\in G$, and likewise $\mathrm{Rep}(A,\mathcal{F})$ is the quotient of
$\mathrm{Hom}(A,S)$, identifying two morphisms if they differ by a morphism in $\mathcal{F}$.

\end{document}